\documentclass{amsart}
\usepackage{amsmath,amscd,amsthm,amsfonts,amssymb}

\theoremstyle{plain}
\newtheorem{theorem}                 {Theorem}      [section]
\newtheorem{free-theorem}            {Theorem}
\newtheorem{proposition}  [theorem]  {Proposition}

\newtheorem{lemma}        [theorem]  {Lemma}

\theoremstyle{definition}
\newtheorem{example}      [theorem]  {Example}

\newtheorem{definition}   [theorem]  {Definition}

\numberwithin{equation}{section}

\def \theo-intro#1#2 {\vskip .25cm\noindent{\bf Theorem #1\ }{\it #2}}

\newcommand{\trace}{\operatorname{trace}}

\def \zn{\mathbb Z}

\def \rn{\mathbb R}
\def \cn{\mathbb C}

\def \Sym{\mathrm{Sym}}

\def \B{\mathcal B}

\def \E{\mathcal E}

\def \Q{\mathcal Q}

\def \W{\mathcal W}
\def \Z{\mathcal Z}

\def \ip #1#2{\langle #1,#2 \rangle}
\def \iq #1#2{( #1,#2 )}

\def \proc#1{\cn P^{#1}}

\def \g{\mathfrak{g}}

\def \t{\mathfrak{t}}

\def \W{\mathcal{W}}

\DeclareMathOperator{\Ad}{Ad}
\DeclareMathOperator{\End}{End}
\DeclareMathOperator{\Hom}{Hom}
\def\span{\mathrm{span}}

\def \GLC#1{\text{\bf GL}_{#1}(\cn)}
\def \glc#1{\mathfrak{gl}_{#1}(\cn)}

\def \slc#1{\mathfrak{sl}_{#1}(\cn)}

\def \SO#1{\text{\bf SO}(#1)}

\def \U#1{\text{\bf U}(#1)}

\def \SU#1{\text{\bf SU}(#1)}
\def \su#1{\mathfrak{su}(#1)}

\def \Sp#1{\text{\bf Sp}(#1)}

\def \nab#1#2{\hbox{$\nabla$\kern -.3em\lower 1.0 ex
 \hbox{$#1$}\kern -.1 em {$#2$}}}

\begin{document}
\baselineskip 22pt \larger

\allowdisplaybreaks

\title
{Complete minimal submanifolds \\ of compact Lie groups}

\author
{Sigmundur Gudmundsson, Martin Svensson and Marina Ville}



\address
{Department of Mathematics, Faculty of Science, Lund University,
Box 118, S-221 00 Lund, Sweden}
\email{Sigmundur.Gudmundsson@math.lu.se}

\address
{Department of Mathematics \& Computer Science, University of
Southern Denmark, Campusvej 55, DK-5230 Odense M, Denmark}
\email{svensson@imada.sdu.dk}

\address
{D\' epartment de Math\' ematiques, Universit\' e F. Rabelais,
  37000 Tours, France}
\email{Marina.Ville@lmpt.univ-tours.fr}

\begin{abstract}
We give a new method for manufacturing complete minimal
submanifolds of compact Lie groups and their
homogeneous quotient spaces.  For this we make use of harmonic
morphisms and basic representation theory of Lie groups.
We then employ our method to construct many examples of compact
minimal submanifolds of the special unitary groups.
\end{abstract}

\maketitle


\section{Introduction}

In this paper we introduce a new method for constructing
complete minimal submanifolds of compact Lie groups and their
homogeneous quotient spaces.  Our most important ingredients
are harmonic morphisms and some basic representation theory of
compact Lie groups.

Complex-valued harmonic morphisms on Riemannian manifolds
are harmonic functions which
additionally satisfy the non-linear condition of horizontal
conformality.  They are in general difficult to find but in
several important cases they can be constructed via the so-called
eigenfamily method described below.  The elements of such a
family are eigenfunctions of
the Laplace-Beltrami operator
diagonalizing a bilinear operator associated with the non-linear
horizontal conformality condition.  This is where representation
theory comes into play.

We apply our method to the standard representations of
the simple Lie groups $\SU n$, $\SO n$ and $\Sp n$.  This yields
well-known eigenfamilies on these space, already constructed
in \cite{Gud-Sak-1}.

Then we focus our attention on the representation $\slc n$
of the special unitary group $\SU n$.  This gives many
{\it new} eigenfamilies on $\SU n$ and its various homogeneous
quotient spaces as flag manifolds.  In the last section we then
construct an interesting continuous
family of minimal submanifolds of $\SU n$, all of codimension $2$:

\begin{theorem}
Let $H$ be an $n\times n$ complex matrix which has $n$ different
eigenvalues. Then the compact subset
$$
M=\{z=(z_1,...,z_n)\in\SU n |\ z_1^tH\bar z_2=0\}
$$
of the special unitary group is a minimal submanifold of codimension
two.
\end{theorem}

For an introduction to representation theory we highly recommend
the excellent text \cite{Ful-Har}.

\section{Harmonic Morphisms}

Let $M$ and $N$ be two manifolds of dimension $m$ and $n$,
respectively. Then a Riemannian metric $g$ on $M$ gives rise
to the notion of a Laplacian on $(M,g)$ and real-valued harmonic
functions $f:(M,g)\to\rn$. This can be generalized to the concept
of a harmonic map $\phi:(M,g)\to (N,h)$ between Riemannian
manifolds being a solution to a semi-linear system of partial
differential equations, see \cite{Bai-Woo-book}.

\begin{definition}\cite{Fug-1},\cite{Ish}
A map $\phi:(M,g)\to (N,h)$ between Riemannian manifolds is
called a {\it harmonic morphism} if, for any harmonic function
$f:U\to\rn$ defined on an open subset $U$ of $N$ with
$\phi^{-1}(U)$ non-empty, the composition
$f\circ\phi:\phi^{-1}(U)\to\rn$ is a harmonic function.
\end{definition}

The following characterization of harmonic morphisms between
Riemannian manifolds is due to Fuglede and Ishihara.
For the definition of horizontal
(weak) conformality we refer to \cite{Bai-Woo-book}.

\begin{theorem}\cite{Fug-1},\cite{Ish}
A map $\phi:(M,g)\to (N,h)$ between Riemannian manifolds is a
harmonic morphism if and only if it is a horizontally (weakly)
conformal harmonic map.
\end{theorem}

The next result of Baird and Eells gives the theory of harmonic
morphisms a strong geometric flavour.  It shows that when the
codomain $N$ is a surface the conditions characterizing harmonic
morphisms are independent of conformal changes of the metric on $N$.

\begin{theorem}\cite{Bai-Eel}\label{theo:B-E}
Let $\phi:(M,g)\to (N^2,h)$ be a horizontally conformal submersion
from a Riemannian manifold to a surface.  Then $\phi$ is harmonic
if and only if $\phi$ has minimal fibres.
\end{theorem}

The following result is very useful when dealing with harmonic
morphism from Lie groups and their homogeneous quotient spaces.

\begin{proposition}\cite{Gud-4}\label{prop:lift}
Let $(M,g)$, $(\hat M,\hat g)$ and $(N,h)$ be Riemannian manifolds.
Furthermore, let $\phi:(M,g)\to (N,h)$ be
a map, $\pi:(\hat M,\hat g)\to(M,g)$ be a surjective
harmonic morphism and $\hat\phi:(\hat M,\hat g)\to(N,h)$ be the
composition  $\hat \phi=\phi\circ\pi$.  Then $\phi$ is a harmonic
morphism if and only if $\hat\phi$ is a harmonic morphism.
\end{proposition}

For the general theory of harmonic morphisms, we refer to the
standard reference \cite{Bai-Woo-book} and the regularly updated
on-line bibliography \cite{Gud-bib}.

\section{The method of eigenfamilies}

Let $\phi,\psi:(M,g)\to\cn$ be functions on a
Riemannian manifold.  Then the metric $g$ induces the complex-valued
Laplacian $\tau(\phi)$ and the gradient $\nabla\phi$ with
values in the complexified tangent bundle $T^{\cn}M$ of $M$.  We
extend the metric $g$ to be complex bilinear on $T^{\cn} M$ and
define the symmetric bilinear operator $\kappa$ by
$$
\kappa(\phi,\psi)= g(\nabla\phi,\nabla\psi).
$$
Two functions $\phi,\psi: M\to\cn$ are said to be {\it orthogonal} if
$\kappa(\phi,\psi)=0$.  With this notation, the harmonicity and horizontal
conformality of $\phi:(M,g)\to\cn$ take the following form
$$
\tau(\phi)=0\ \ \text{and}\ \ \kappa(\phi,\phi)=0.
$$

\begin{definition}\cite{Gud-Sak-1}\label{definition:eigenfamily}
Let $(M,g)$ be a Riemannian manifold.  Then a set
$$
\E =\{\phi_i:M\to\cn\ |\ i\in I\}
$$
of complex-valued functions is said to be an {\it eigenfamily} for
$M$ if there exist complex numbers $\lambda,\mu\in\cn$ such that for
all $\phi,\psi\in\E$
$$
\tau(\phi)=\lambda\phi\ \ \text{and}\ \ \kappa(\phi,\psi)=\mu\phi\psi.
$$
\end{definition}

The next result is a reformulation of Theorem 2.5 of \cite{Gud-Sak-1}.
It shows that an eigenfamily for a Riemannian manifold can be used
to produce a large variety of harmonic morphisms.

\begin{theorem}\label{theo:rational}
Let $(M,g)$ be a Riemannian manifold and $\E =\{\phi_1,\dots
,\phi_n\}$ be a finite eigenfamily of complex-valued functions on
$M$. If $P,Q:\cn^n\to\cn$ are linearly independent homogeneous
polynomials of the same degree then $\phi:U\to\proc 1$ with
$$
\phi=[P(\phi_1,\dots ,\phi_n),Q(\phi_1,\dots ,\phi_n)]
$$
is a non-constant harmonic morphism on the open and dense subset
$$
U=\{p\in M| \ P(\phi_1(p),\dots ,\phi_n(p))\neq 0\ \
\text{or}\ \ Q(\phi_1(p),\dots ,\phi_n(p))\neq 0\}.
$$
\end{theorem}

\section{Complete Minimal Submanifolds}\label{section-submanifolds}

Let $(M,g)$ be a Riemannian manifold and $\E =\{\phi_1,\dots,\phi_n\}$
be a finite eigenfamily of complex-valued functions on
$M$. Let $P,Q:\cn^n\to\cn$ be linearly independent homogeneous
polynomials of the same degree. For each non-zero
$\hat\xi=(\alpha,\beta)\in\cn^2$ and $\xi=[\alpha,\beta]\in\proc 1$
we define $M_\xi$  by
$$
M_\xi=\{p\in M|\ \beta\cdot P(\phi_1(p),\dots ,\phi_n(p))
-\alpha\cdot Q(\phi_1(p),\dots ,\phi_n(p))=0\}.
$$
Then it is clear that $M_\xi$ is a complete subset of
$$
M=\bigcup_{\xi\in\proc 1}M_{\xi}
$$
and if $\xi_1,\xi_2\in\proc 1$ are different then
$M_{\xi_1}\cap M_{\xi_2}=\Z$ where
$$
\Z=\{p\in M|\ P(\phi_1(p),\dots ,\phi_n(p))=0\
\text{and}\ Q(\phi_1(p),\dots ,\phi_n(p))=0\}.
$$
For the above situation we define the map $\Psi_{\hat\xi}:M\to\cn$ by
$$
\Psi_{\hat\xi}(p)=\beta\cdot P(\phi_1(p),\dots ,\phi_n(p))
-\alpha\cdot Q(\phi_1(p),\dots ,\phi_n(p)).
$$
The implicit function theorem tells us that if $0\in\cn$ is a
regular value of $\Psi_{\hat\xi}$ then the inverse image
$M_\xi=\Psi_{\hat\xi}^{-1}(\{0\})$ is a submanifold of $M$ of codimension two.
In that case it follows from Theorems \ref{theo:rational} and
\ref{theo:B-E} that the open and dense subset $M_\xi\setminus\Z$
of $M_\xi$ is minimal in $M$.  Hence $M_\xi$ is a complete minimal
submanifold in $M$.  This gives us an attractive method for producing
complete minimal submanifolds of Riemannian manifolds.
In Section \ref{section-example-min-submanifolds}, we apply this to the special
unitary groups $\SU n$, after elaborating on the needed
harmonic morphisms.

\section{Some useful Representation Theory}

Let $G$ be a compact Lie group equipped with a bi-invariant inner
product. Let $\g$ be the Lie algebra of $G$ and fix a maximal torus $T$ in
$G$ with Lie algebra $\t$. Denote by $\Lambda_W$ the weight lattice in
$(\t^\cn)^*$ and fix a dominant Weyl chamber $\W\subset(\t^\cn)^*$; for any
$\lambda\in\overline\W\cap\Lambda_W$, denote by $V_\lambda$ the irreducible
representation of $\g^\cn$ with highest weight $\lambda$. Recall that this
representation lifts to an irreducible representation of $G$ if and only
if $\lambda$ is analytically integral.

According to the Peter-Weyl theorem we have an orthogonal decomposition
$$
L^2(G)=\bigoplus_{\lambda}M(V_\lambda),
$$
where the sum is taken over all analytically integral dominant weights of
$G$ and $M(V_\lambda)$ denotes the space spanned by the matrix elements of
the representation.

We fix a $G$-invariant Hermitian inner product on $V_\lambda$. When the
representation is of real (quaternionic) type we also fix an invariant
symmetric (skew-symmetric) bi-linear form on $V_\lambda$. Consider a
function $\phi:G\to\cn$ of the form
$$
\phi(g)=q(ga,b)\qquad(a,b\in V_\lambda),
$$
where $q$ is any non-degenerate invariant bi-linear or Hermitian form on
$V_\lambda$. Such a function is an eigenfunction of the Laplacian on $G$ since
$$
\tau(\phi(g))=\sum_{X\in\B}q(gX^2a,b)=q(gCa,b),
$$
where $\B$ is an orthonormal basis for $\g$ and
$$
C=\sum_{X\in\B}X^2
$$
is the Casimir element in the universal enveloping algebra of $\g^\cn$.
As the representation is irreducible, $C$ acts on $V_\lambda$ as the scalar
$$
\alpha_\lambda=-(|\lambda|^2+2\ip{\lambda}{\delta}),
$$
where $\delta$ is half the sum of the positive roots, see Proposition 5.28
of \cite{Kna}. Hence
$$
\tau(\phi)=\alpha_\lambda\phi.
$$
In particular, all the functions in $M(V_\lambda)$ are eigenfunctions of the
Laplacian, all with the same eigenvalue $\alpha_\lambda$.

Concerning the $\kappa$-operator, let $\phi,\psi:G\to\cn$
be given by $\phi(g)=q(ga,b)$,
$\psi(g)=q(gu,v)$ where $a,b,u,v\in V_\lambda$.  Then
$$
\kappa(\phi(g),\psi(g))=\sum_{X\in\B}q(gXa,b)q(gXu,v).
$$
Hence we have a good reason to consider the map
$$
Q(a,b,c,d)=\sum_{X\in\B}q(Xa,b)q(Xc,d).
$$
This map is clearly $G$-invariant and its interpretation depends on the
type of the form $q$ on $V_\lambda$:
\begin{enumerate}
\item If $q=\ip{\cdot}{\cdot}$ is a Hermitian form, then $Q$ may be thought
of as a self-adjoint $G$-equivariant endomorphism on
$V_\lambda\otimes V_\lambda$ given by
$$
\ip{Q(a\otimes c)}{b\otimes d}=Q(a,b,c,d).
$$
Alternatively, we can think of $Q$ as a $G$-equivariant endomorphism on
$V_\lambda\otimes V_\lambda^*\cong\End(V_\lambda)$, defined by
$$
\ip{Q(a\otimes b)}{c\otimes d}=Q(a,b,c,d).
$$
\item If $q=\iq{\cdot}{\cdot}$ is a symmetric bi-linear form (in which case
the representation is of \emph{real type}), then $Q$ may be thought of as a
self-adjoint $G$-equivariant endomorphism on $\Lambda^2V_\lambda$ given by
$$
\iq{Q(a\wedge b)}{c\wedge d}=Q(a,b,c,d).
$$
\end{enumerate}
Let $W$ denote either $V_\lambda\otimes V_\lambda$, $V_\lambda\otimes V_\lambda^*$
or $\Lambda^2V_\lambda$. For any irreducible subrepresentation $W'$ of $W$, $Q$
restricts to a $G$-equivariant endomorphism on $W'$, which, by Schur's Lemma,
must be a multiple of the identity endomorphism on $W'$. Hence there is a
scalar $\mu$, such that
$Q=\mu q$ on $W'$.

\section{The standard representation $\cn^n$ of $\SU n$}
\label{section-SUn}

Consider the standard representation $\cn^n$ of $\SU n$, equipped
with the standard Hermitian inner product $\ip{\cdot}{\cdot}$, and $Q$
as a self-adjoint map $$Q:\cn^n\otimes\cn^n\to\cn^n\otimes\cn^n.$$  Since
$$
\cn^n\otimes\cn^n=\Sym^2(\cn^n)\oplus\Lambda^2\cn^n
$$
and $\Sym^2(\cn^n)$ is irreducible, we restrict $Q$ to $\Sym^2(\cn^n)$.
If $a,b\in\mathbb{C}^n$, we denote by $a\cdot b$ the image of
$a\otimes b$ in $\Sym^2(\cn^n)$.

Following Schur's Lemma this is a scalar multiple $\mu$ of the identity i.e.
$$
\ip{Q(a\cdot b)}{c\cdot d}=\mu\ip{a\cdot b}{c\cdot d}
=\mu(\ip{a}{c}\ip{b}{d}+\ip{a}{d}\ip{b}{c}).
$$
In particular, we have
\begin{equation}\label{eq:Q-and-mu}
\ip{Q(a\cdot a)}{c\cdot d}=2\mu\ip{a}{c}\ip{a}{d}.
\end{equation}

For $z\in\SU n$, we let $\phi_c(z)=\ip{za}{c}$. As a direct consequence
of (\ref{eq:Q-and-mu}) we see that the $\kappa$-operator satisfies
\begin{eqnarray*}
\kappa (\phi_c(z),\phi_d(z))
&=&\sum_{X\in\B}\ip{zXa}{c}\ip{zXa}{d}\\
&=&\sum_{X\in\B}\ip{Xa}{z^{-1}c}\ip{Xa}{z^{-1}d}\\
&=&\ip{Q(a\cdot a)}{z^{-1}c\cdot z^{-1}d}\\
&=&\mu\ip{a\cdot a}{z^{-1}c\cdot z^{-1}d}\\
&=&\mu(\ip{a}{z^{-1}c}\ip{a}{z^{-1}d}+\ip{a}{z^{-1}c}\ip{a}{z^{-1}d})\\
&=&2\mu\ip{za}{c}\ip{za}{d}\\
&=&2\mu\ \phi_c(z)\phi_d(z).
\end{eqnarray*}
These calculations show that for any fixed non-zero element $a\in\cn^n$
the set
$$\E_a=\{\phi_c:\SU n\to\cn\ |\ \phi_c(z)=\ip{za}{c},\ c\in\cn^n\}$$
is an eigenfamily on $\SU n$. This was already constructed in
Theorem 5.2 of \cite{Gud-Sak-1} using a different approach.
The stabilizer subgroup of $\SU n$ fixing the element $a$ is isomorphic
to $\SU {n-1}$ so $\E_a$ induces an eigenfamily on the odd-dimensional
sphere $$S^{2n-1}=\SU n/\SU{n-1}.$$  The induced local harmonic morphisms
live on the complex projective space
$$\cn P^{n-1}=\SU n/\text{\bf S}(\U 1\times\U{n-1}).$$
They are clearly holomorphic with respect to the standard K\" ahler
structure.

\begin{example}
Any non-zero element $a\in\cn^2$ induces the following
eigenfamily of complex valued functions
$$\E_a=\{\phi_c:\SU 2\to\cn\ |\ \phi_c(z)=\ip{za}{c},\ c\in\cn^2\}.$$
For linearly independent $c,d\in\cn^2$ and non-zero
$\hat\xi=(\alpha,\beta)\in\cn^2$
define $\Psi_{\hat\xi}:\SU 2\to\cn$ by
$$\Psi_{\hat\xi}:z\mapsto \beta\ip{za}{c}-\alpha\ip{za}{d}.$$
Then
$\Z=\{z\in\SU 2|\ \ip {za}c= 0\ \ \text{and}\ \ \ip{za}d= 0\}$
is empty so we have a globally defined harmonic morphism
$\phi:\SU 2\to\proc 1$ given by
$$\phi:z\mapsto [\ip{za}c,\ip{za}d].$$
The fibres of this map are the well-known compact Hopf circles in $\SU 2$.
\end{example}

\section{The standard representation $\cn^n$ of $\SO n$}
\label{section-SOn}

Consider the standard representation $\cn^n$ of $\SO n$, equipped
with the standard bi-linear form $\iq{\cdot}{\cdot}$, and $Q$ as
a self-adjoint map $$\cn^n\otimes\cn^n\to\cn^n\otimes\cn^n.$$ Since
$$
\cn^n\otimes\cn^n=\Sym^2(\cn^n)\oplus\Lambda^2\cn^n
$$
and $\Lambda^2\cn^n$ is irreducible, we restrict $Q$ to $\Lambda^2\cn^n$.
Following Schur's lemma this is a scalar multiple $\mu$ of the identity i.e.
$$
\iq{Q(a\wedge b)}{c\wedge d}=\mu\iq{a\wedge b}{c\wedge d}
=\mu(\iq{a}{c}\iq{b}{d}-\iq{a}{d}\iq{b}{c}).
$$
In particular, we have
$$
\iq{Q(a\wedge b)}{a\wedge d}=\mu(\iq{a}{a}\iq{b}{d}-\iq{a}{d}\iq{b}{a}).
$$
For a fixed isotropic element $a\in\cn^n$ we now see that the
$\kappa$-operator satisfies
\begin{eqnarray*}
\kappa (\phi_b(x),\phi_d(x))
&=&\sum_{X}\iq{xXa}{b}\iq{xXa}{d}\\
&=&\sum_{X}\iq{Xa}{x^{-1}b}\iq{Xa}{x^{-1}d}\\
&=&\iq{Q(a\wedge x^{-1}b)}{a\wedge x^{-1}d}\\
&=&\mu\iq{a\wedge x^{-1}b}{a\wedge x^{-1}d}\\
&=&\mu (\iq aa\iq {x^{-1}b}{x^{-1}d} - \iq a{x^{-1}b}\iq a{x^{-1}d})\\
&=&-\mu\ \phi_b(x)\phi_d(x).
\end{eqnarray*}
This shows that for a fixed isotropic element $a\in\cn^n$ the following
set of complex-valued functions is an eigenfamily on $\SO n$
$$\E_a=\{\phi_b:\SO n\to\cn\ |\ \phi_b(x)=\iq{xa}{b},\ b\in\cn^n\}.$$
These are exactly those constructed in Theorem 4.3 of \cite{Gud-Sak-1}.
The stabilizer subgroup of $\SO n$ fixing the isotropic $a\in\cn^n$
is isomorphic to
$\SO{n-2}$ so the induced local harmonic morphisms live
on the complex quadric $$\Q_{n-2}=\SO n/(\SO 2\times\SO{n-2}).$$
These maps are holomorphic with respect to the standard complex
structure on $\Q_{n-2}$ induced by the holomorphic embedding
$\Q_{n-2}\hookrightarrow\cn P^{n-1}$.

\section{The standard representation $\cn^{2n}$ of $\Sp n$}
\label{section-Spn}

Consider the standard representation $\cn^{2n}$ of $\Sp n$ equipped
with the standard Hermitian inner product $\ip{\cdot}{\cdot}$.
The form $Q$ defines a self-adjoint map
$Q:\cn^{2n}\otimes\cn^{2n}\to\cn^{2n}\otimes\cn^{2n}$ by
$$
\ip{Q(a\otimes c)}{b\otimes d}=Q(a,b,c,d)
$$
Since
$$
\cn^{2n}\otimes\cn^{2n}=\Sym^2(\cn^{2n})\oplus\Lambda^2\cn^{2n}
$$
and $\Sym^2\cn^{2n}$ is irreducible, we consider the restriction
of $Q$ to $\Sym^2\cn^{2n}$.  Following Schur's lemma this is a scalar
multiple $\mu$  of the identity i.e.
$$
\ip{Q(a\cdot c)}{b\cdot d}=\mu\ip{a\cdot c}{b\cdot d}
=\mu(\ip{a}{b}\ip{c}{d}+\ip{a}{d}\ip{c}{b}).
$$
In particular, we have
$$
Q(a,b,a,d)=\ip{Q(a^2)}{b\cdot d}=2\mu\ip{a}{b}\ip{a}{d}.
$$
As a direct consequence we see that the $\kappa$-operator satisfies
\begin{eqnarray*}
\kappa (\phi_c(q),\phi_d(q))
&=&\sum_{X\in\B}\ip{qXa}{c}\ip{qXa}{d}\\
&=&\sum_{X\in\B}\ip{Xa}{q^{-1}c}\ip{Xa}{q^{-1}d}\\
&=&Q(a,q^{-1}c,a,q^{-1}d)\\
&=&2\mu\ip{a}{q^{-1}c}\ip{a}{q^{-1}d}\\
&=&2\mu\ \phi_c(q)\phi_d(q).
\end{eqnarray*}
This shows that for any fixed element $a\in\cn^{2n}$ the following set
of complex-valued funtions is an eigenfamily on $\Sp n$
$$
\E_a=\{\phi_c:\Sp n\to\cn\ |\ \phi_c(q)=\ip{qa}{c},\ c\in\cn^{2n}\}.
$$
These are exactly those constructed in Theorem 6.2 in \cite{Gud-Sak-1}.
For a given element $a\in\cn^{2n}$ the stabilizer subgroup of $\Sp n$
fixing $a$ is isomorphic to $\Sp{n-1}$ so the induced local
harmonic morphisms live on the sphere $$S^{4n-1}=\Sp n/\Sp{n-1}.$$

\section{The dual representation $(\cn^n)^*$ of $\SU n$}

To the standard representation $\cn^n$ of $\SU n$ we have the
dual representation $(\cn ^n)^*$ given by
$$
\cn^n\ni b\mapsto\ip{\cdot}{b}\in (\cn^n)^*.
$$
A calculation, similar to that above, shows that for any fixed
non-zero element $a\in\cn^n$ the set
$$\E^*_a=\{\phi_c:\SU n\to\cn\ |\ \phi_c(z)=\ip{c}{za},\ c\in\cn^n\}$$
is an eigenfamily on $\SU n$. The stabilizer subgroup of $\SU n$ fixing
$a$ is isomorphic to $\SU {n-1}$ so $\E_a^*$ induces an
eigenfamily on $S^{2n-1}$  and the induced local harmonic morphisms
live on $\cn P^{n-1}$.  They are clearly anti-holomorphic with respect
to the standard K\" ahler structure.

\section{The representation $\slc n$ of $\SU n$}
\label{section-slcn}

As representations of $\SU n$, we can identify the tensor product
$\cn^n\otimes(\cn^n)^*$ with $\Hom(\cn^n,\cn^n)$ by declaring
$a\otimes b$ to correspond to the linear map
$$
\cn^n\ni v\mapsto \ip{v}{b}a\in\cn^n.
$$
The representation $\cn^n\otimes(\cn^n)^*$ of $\SU n$ corresponds
then to the standard adjoint representation of $\SU n$ on
$\Hom(\cn^n,\cn^n)$ i.e. $z\cdot A=zAz^{-1}$. The standard invariant
inner product
$$
\ip{A}{B}=\trace(A\cdot B^*)
$$
on $\Hom(\cn^n,\cn^n)$ then translates to the product
$$
\ip{a\otimes b}{c\otimes d}=\ip{a}{c}\ip{d}{b}
$$
on $\cn^n\otimes(\cn^n)^*$ and they are obviously invariant under $\SU n$.
Furthermore,
$$
(a\otimes b)^*=b\otimes a.
$$

The representation $\Hom(\cn^n,\cn^n)$ of $\SU n$ decomposes into
irreducible subrepresentations
$$
\Hom(\cn^n,\cn^n)=\span\{I\}\oplus\slc n
$$
where $\slc n$ are the trace-free endomorphisms.
For $A,B\in\slc n$ we define
$$
\phi(z)=\ip{z\cdot A}{B}=\ip{zAz^{-1}}{B}\qquad(z\in\SU n).
$$
Since the representation is irreducible, we know that $\phi$
is an eigenfunction of the Laplacian.

To study the $\kappa$-operator, we fix an orthonormal basis $\B$
of $\su n$ and consider the map
$$
Q(A,B,C,D)=\sum_{X\in B}\ip{[X,A]}{B}\ip{[X,C]}{D}\qquad (A,B,C,D\in\slc n).
$$
If $\phi(z)=\ip{z\cdot A}{B}$ and $\psi(z)=\ip{z\cdot C}{D}$, then we have
$$
\kappa(\phi(z),\psi(z))=Q(z\cdot A,B,z\cdot C,D).
$$
Now, it follows easily that
$$
\ip{[A,B]}{C}=\ip{A}{[C,B^*]}.
$$
Furthermore, since $\slc n$ is the complexification of $\su n$, $\B$
is a Hermitian basis of $\slc n$. Hence
\begin{eqnarray*}
Q(A,B,C,D)&=&\sum_{X\in\B}\ip{[X,A]}{B}\ip{[X,C]}{D}\\
&=&\sum_{X\in\B}\ip{X}{[B,A^*]}\ip{X}{[D,C^*]}\\
&=&\sum_{X\in\B}\ip{X}{[B,A^*]}\ip{[D^*,C]}{X}\\
&=&\ip{[D^*,C]}{[B,A^*]}.
\end{eqnarray*}
Let us now write
$A=a\otimes b,\ B=c\otimes d,\ C=e\otimes f,\ D=g\otimes h$
and assume that
$a\perp b,\ e\perp f$,
which also ensures that $A,C$ are trace-free. Then
\begin{eqnarray*}
\ip{[D^*,C]}{[B,A^*]}&=&\ip{[h\otimes g,e\otimes f]}
{[c\otimes d,b\otimes a]}\\
&=&\ip{\ip{e}{g}h\otimes f-\ip{h}{f}e\otimes g}
{\ip{b}{d}c\otimes a-\ip{c}{a}b\otimes d}\\
&=&\ip{e}{g}\ip{d}{b}\ip{h}{c}\ip{a}{f}
-\ip{e}{g}\ip{a}{c}\ip{h}{b}\ip{d}{f}\\
& &-\ip{h}{f}\ip{d}{b}\ip{e}{c}\ip{a}{g}
+\ip{h}{f}\ip{a}{c}\ip{e}{b}\ip{d}{g}.
\end{eqnarray*}
If we now assume that $a,b,e,f$ are all mutually orthogonal,
then this reduces to
\begin{eqnarray*}
\ip{[D^*,C]}{[B,A^*]}&=&
-\ip{e}{g}\ip{a}{c}\ip{h}{b}\ip{d}{f}\\
& &-\ip{h}{f}\ip{d}{b}\ip{e}{c}\ip{a}{g}.
\end{eqnarray*}

Further $A=C$ i.e. $a=e$, $b=f$ give
$Q(A,B,A,D)=-2\ip{A}{B}\ip{A}{D}$, hence
$$
\kappa(\phi(z),\psi(z))=Q(z\cdot A,B,z\cdot A,D)=-2\ \phi(z)\psi(z).
$$
In fact, we see that by fixing $a,b\in\cn^n$ with $a\perp b$, then
$$
\{\ip{z\cdot(a\otimes b)}{c\otimes d}\ |\ c,d\in\cn^n\}
=\{\ip{z\cdot a}{c}\ip{d}{z\cdot b}\ |\ c,d\in\cn^n\}
$$
is an eigenfamily.  As a direct consequence we have the following
far going generalization of Example 2.3 of \cite{Gud-4}.

\begin{example}\label{exam-tensor}
Equip $\cn^n$ with its standard Hermitian scalar product and let
$\{ e_1,e_2, \dots ,e_n\}$ be some orthonormal basis.
Then define the complex-valued functions $\phi_{ik}:\SU n\to\cn$ by
$$\phi_{ik}(z)=\ip{z\cdot e_1}{e_i}\ip{e_k}{z\cdot e_2}
=\ip{z_1}{e_i}\ip{e_k}{z_2}=z_{i1}\bar z_{k2}.$$
It then follows from the above calculations that these functions
generate the following complex $n^2$-dimensional eigenfamily on $\SU n$
$$\E=\{\phi_A:\SU n\to\cn\,|\ \phi_A(z)=z_1^tA\bar z_2,\ A\in\cn^{n\times n}\}.$$
Let $K\cong\text{\bf S}(D\times\U{n-2})$ be the stabilizer
subgroup of $\SU n$ fixing $e_1\otimes e_2^*$.  Here $D\cong\U 1$
denotes the diagonal of $\U 2$.   The set $\E$ induces an
eigenfamily on the homogeneous quotient space
$\SU n/\text{\bf S}(D\times\U{n-2})$ and the harmonic morphisms,
induced by this, actually live on the flagmanifold
$$F=\SU n/\text{\bf S}(\U 1\times\U 1\times\U{n-2}).$$
They are not holomorphic with respect to the standard K\" ahler
structure on $F$ induced by the standard K\" ahler structure on
the complex Grassmannian of $2$-planes in $\cn^n$ via the
homogeneous projection map
$$F\to\SU n/\text{\bf S}(\U 2\times\U{n-2}).$$
\end{example}

Here we have an alternative proof of the construction in
Example \ref{exam-tensor}, in the spirit of the \cite{Gud-Sak-1}.

\begin{proof}
Let $\{ e_1,e_2, \dots ,e_n\}$ be some orthonormal basis for $\cn^n$
and define the complex-valued functions $\phi_{i},\psi_{k}:\SU n\to\cn$
on the special unitary group by
$$  \phi_{i}(z)=\ip{z\cdot e_1}{e_i}=z_{i1},
\ \ \psi_{k}(z)=\ip{e_k}{z\cdot e_2}=\bar z_{k2}.$$
It then follows from above that
$$\E_1=\{\phi_{i} |\ i=1,2,\dots,n\}\ \ \text{and}\ \
\E_2=\{\psi_{k} |\ k=1,2,\dots,n\}$$
are two eigenfamilies and a simple calculation shows that
$\kappa(\phi_i,\psi_k)=0$.
As a direct consequence of Lemma A.1 in \cite{Gud-Sak-1} we see that
$$\E=\{\phi_{i}\psi_{k}|\ i,k=1,2,\dots,n\}$$ is also an
eigenfamily of $\SU n$.
\end{proof}

With the following example we now generalize further.

\begin{example}\label{exam-tensor-extended}
Equip $\cn^n$ with its standard Hermitian scalar product and let
$\{ e_1,e_2, \dots ,e_n\}$ be some orthonormal basis.  Then
define the complex-valued functions $\phi_{ijkl}:\SU n\to\cn$ by
$$\phi_{ijkl}(z)=\ip{z\cdot e_j}{e_i}\ip{e_k}{z\cdot e_l}
=z_{ij}\bar z_{kl},\ \ \text{where}\ \  i,j,k,l=1,\dots,n.$$
It then follows from the above calculations that, for each
$s\in\zn^+$ with $2s\le n$, we have an eigenfamily $\E_{s}$
on $\SU n$ given by
$$\E_{s}=\{\,  \sum_{r=1}^sz_{2r-1}^tA_r\bar z_{2r}\, |
\ A_r\in\cn^{n\times n}\ \text{and}\ r=1,2,\dots,s\},$$
which clearly is a complex vector space of dimension $n^{2s}$.  Let
$$K\cong\text{\bf S}(D\times\dots\times D\times\U{n-2s})$$
be the stabilizer subgroup of $\SU n$
fixing $$e_1\otimes e_2^*, \dots, e_{2s-1}\otimes e_{2s}^*.$$
Here $D\cong\U 1$ denotes the diagonal of $\U 2$.
The set $\E_{s}$ induces an eigenfamily on the homogeneous quotient space
$\SU n/\text{\bf S}(D\times\U{n-2})$ and the harmonic morphisms,
induced by this, actually live on the flagmanifold
$$\SU n/\text{\bf S}(\U 1\times\cdots\times\U 1\times\U{n-2s}).$$
\end{example}

\section{Compact minimal submanifolds in $\SU n$}
\label{section-example-min-submanifolds}
In this section we use the previous constructions to produce
smooth closed minimal submanifolds of $\SU n$.
If $z$ is a matrix in $\SU n$ then we denote its columns
by $z_1,\dots ,z_n$ and write $z=(z_1,\dots ,z_n)$.
\begin{theorem}\label{theo-examples}
Let $H$ be an $n\times n$ complex matrix which has $n$
different eigenvalues. Then the compact subset
$$
M=\{(z_1,\dots,z_n)\in\SU n |\ z_1^tH\bar z_2=0\}
$$
of the special unitary group is a minimal submanifold of
codimension two.
\end{theorem}

\begin{proof}
We begin by showing that $M$ is smooth. To this effect, we let
$$\Phi(z)=z_1^tH\bar z_2=\ip{\Ad_z(H)e_1}{e_2}=\ip{\Ad_z(H)}{e_2\otimes e_1}
=\ip H{z^{-1}(e_2\otimes e_1)}.
$$
Note that if $e_2\otimes e_1$ represents the endomorphism $A$, then
$$z^{-1}(e_2\otimes e_1)=z^{-1}(e_2)\otimes z^{-1}( e_1)=z^{-1}Az.$$
This implies that the gradient of $\Phi$ satisfies
\begin{equation}\label{eq:gradient}
\nabla\Phi(z)=\sum_{X\in B}\ip{[X,H]}{z^{-1}(e_2\otimes e_1)}X,
\end{equation}
where $\B$ is any orthonormal basis for the Lie algebra $\su n$.
Our goal is thus to find a vector $X$ in $\glc n$
such that
\begin{equation}
\ip{[X,H]}{z^{-1}(e_2\otimes e_1)}\neq 0.
\end{equation}
Since $H$ has $n$ distinct eigenvalues, there exists a diagonal matrix
$D$ and a matrix $P\in \GLC n$ such that $$H=PDP^{-1}.$$
We let $\{\epsilon_1,\dots ,\epsilon_n\}$ be a basis with respect to which $D$ is
diagonal and we denote its eigenvalues by $\lambda_1,\dots ,\lambda_n$ i.e. we have
\begin{equation}\label{eq:diagonal matrix}
D=\sum_{i=1}^n \lambda_i \epsilon_i\otimes \epsilon_i.
\end{equation}
Thus, if $X\in \slc n$ then
$[X,H]=P[P^{-1}XP,D]P^{-1}$ and
\begin{eqnarray*}
\ip{[X,H]}{z^{-1}(e_2\otimes e_1)}
&=&\ip{P[P^{-1}XP,D]P^{-1}}{z^{-1}(e_2\otimes e_1)}\\
&=&\trace (P[P^{-1}XP,D]P^{-1}(z^{-1}(e_2\otimes e_1)^*))\\
&=&\trace ([P^{-1}XP,D]P^{-1}(z^{-1}(e_2\otimes e_1)^*)P)\\
&=&\ip{[P^{-1}XP,D]}{P^*z^{-1}(e_2\otimes e_1)(P^{-1})^*}\\
&=&\ip{[P^{-1}XP,D]}{P^*z^{-1}(e_2\otimes e_1)(P^*)^{-1}}.
\end{eqnarray*}
Because $e_1$ and $e_2$ are orthogonal in $\mathbb{C}^n$, $z^{-1}(e_1)$
and $z^{-1}(e_2)$ are also orthogonal, hence the endomorphisms
$z^{-1}(e_2\otimes e_1)$ and $P^*z^{-1}(e_2\otimes e_1)(P^*)^{-1}$
are nilpotent. On the other hand,
$P^*z^{-1}(e_2\otimes e_1)(P^*)^{-1}$ is not identically zero,
hence it is not diagonal and there exist $i,j$ such that $i\neq j$ and
\begin{equation}\label{eq:non diagonal}
\ip{\epsilon_i\otimes \epsilon_j}{P^*z^{-1}(e_2\otimes e_1)(P^*)^{-1}}\neq 0.
\end{equation}

We let $X_0=P(\epsilon_i\otimes \epsilon_j)P^{-1}$ and obtain
$$[X_0,D]=[\epsilon_i\otimes \epsilon_j,D]
=(\lambda_i- \lambda_j)\epsilon_i\otimes \epsilon_j.$$
Since $\lambda_i\neq\lambda_j$ it follows that
$$\ip{[X_0,H]}{z^{-1}(e_2\otimes e_1)}\neq 0.$$
Using Lemma \ref{lemm-gradient vanishes}, we derive that $\nabla\Psi$
is non-zero, hence $M$ is smooth.

To prove that $M$ is minimal, we choose two linearly independent
$A,B\in\cn^{n\times n}$ such that
$H=A-B$.  It then follows from Example \ref{exam-tensor} and
Theorem \ref{theo:rational} that the map $\phi:U\to\proc 1$ with
$$
\phi(z)=[z_1^tA\bar z_2,z_1^tB\bar z_2]
$$
is a non-constant harmonic morphism on the open and dense subset
$$
U=\{z\in\SU n| \ z_1^tA\bar z_2\neq 0\ \ \text{or}\ \ z_1^tB\bar z_2\neq 0\}.
$$
The fibre $\phi^{-1}(\{[1,1]\})$ in $\SU n$ is a dense open subset of $M$;
it is moreover minimal in $\SU n$ by Theorem \ref{theo:B-E}. It follows that $M$ is minimal.
\end{proof}

\begin{lemma}\label{lemm-gradient vanishes}
Assume that the gradient $\nabla\Psi$ vanishes at a point $z\in \SU n$. Then,
for all $X\in{\mathfrak gl}(n,\mathbb{C})$
\begin{equation}\label{eq:bracket}
\ip{[X,H]}{z^{-1}(e_2\otimes e_1)}=0.
\end{equation}
\end{lemma}

\begin{proof}
First we notice that if $X=\lambda I$ for some complex number
$\lambda$, then $[X,H]=0$ and (\ref{eq:bracket}) is automatically verified.
So we can assume that $X\in{\mathfrak sl}(n,\mathbb{C})$.
If $X$ belongs to $\su n$, then equation (\ref{eq:bracket}) follows immediately
from (\ref{eq:gradient}).
To treat the general case, we notice that $\Psi$ is the restriction
to $\SU n$ of the holomorphic map
$\tilde{\Psi}:{\bf SL}(n,\mathbb{C})\longrightarrow \mathbb{C}$ with
$$z\mapsto \ip H{z^{-1}(e_2\otimes e_1)}.$$
If $X\in\mathfrak{sl}(n,\mathbb{C})$, there exist $X_1,X_2$ in $\su n$
such that $X=X_1+iX_2$. Since $\tilde{\Psi}$ is holomorphic,
$$d\tilde{\Psi}(X)=d\tilde{\Psi}(X_1)+id\tilde{\Psi}(X_2)=d\Psi(X_1)+id\Psi(X_2)=0.$$
\end{proof}

\end{document}